\documentclass[a4paper]{article}
\usepackage{amsmath}
\usepackage{amssymb}
\usepackage{amsthm}
\usepackage{dsfont}
\usepackage[notref,notcite]{showkeys}
 \newtheorem{thm}{Theorem}
 \newtheorem{prop}{Proposition}
 \newtheorem{remark}{Remark}
 \newtheorem{lemma}{Lemma}
 \newcommand{\bbr}{{\mathbb R}}
\newcommand{\bbs}{{\mathbb S}}

\usepackage{authblk}

\def\be{\begin{equation}}
\def\ee{\end{equation}}

\begin{document}
\title{The spatially homogeneous Boltzmann equation for massless particles in an FLRW background}

\author{Ho Lee\footnote{holee@khu.ac.kr}}

\affil{Department of Mathematics and Research Institute for Basic Science, Kyung Hee University, Seoul, 02447, Republic of Korea}

\maketitle

\begin{abstract}
We study the spatially homogeneous relativistic Boltzmann equation for massless particles in an FLRW background with scattering kernels in a certain range of soft and hard potentials. We obtain the future global existence of small solutions in a weighted $L^1\cap L^\infty$ space. 
\end{abstract}

%
%

\section{Introduction}
The relativistic Boltzmann equation describes the time evolution of the distribution function for fast-moving particles undergoing binary collisions. It was first considered in general relativity \cite{B73,BCB,B67}, but detailed analysis of the collision operator was initiated in special relativity \cite{DEJ,DEJ92,GS,GS95}. There have been many works on the relativistic Boltzmann equation \cite{A96,ACI,G06,GS12,JY,LR,S10,SY,SZ}, but still many open problems remain to be investigated. For basic information about the relativistic Boltzmann equation we refer to \cite{CK,G}. 

In this paper, we are interested in the relativistic Boltzmann equation, but it will be studied in an FLRW\footnote{Friedmann-Lema{\^i}tre-Robertson-Walker} background. To be consistent with the FLRW geometry, we assume that the distribution function is also spatially homogeneous. There have been only a few results concerning the spatially homogeneous relativistic Boltzmann equation in an FLRW spacetime. Global existence of small solutions was obtained in \cite{L13}, but a certain restriction was imposed on the angular part of the scattering kernel. The restriction was removed in \cite{LN1}, but the argument applies only to the case of the scattering kernel for Israel particles. These results have been extended to the Bianchi cases \cite{LN,LN2,LN3}, and we also refer to \cite{ND06,NDT05,NT06} for a different approach. The purpose of this paper is to study the global existence of small solutions to the spatially homogeneous relativistic Boltzmann equation in an FLRW background, and the results of the paper will be an improvement of \cite{L13,LN1}, in the sense that the unphysical restriction of \cite{L13} will be removed, and the global existence will be proved for a wider class of scattering kernels. 

On the other hand, we will consider the Boltzmann equation for massless particles, which differs from the massive cases of \cite{L13,LN1}. In this paper, the spatially homogeneous relativistic Boltzmann equation for massless particles will be referred to as the massless Boltzmann equation, for simplicity. The massless Boltzmann equation has recently been studied in \cite{baz1,baz2}, where an analytic solution has been found, and in \cite{LNT}, where a local existence was obtained. The main interest in \cite{LNT} was to study the isotropic singularity problem, for which one needs to establish a well-posed Cauchy problem with data at $t=0$, where the initial singularity is located (see \cite{A,AT99a,AT99,T03,T} for more details). In this paper, we also study the Cauchy problem for the massless Boltzmann equation, but $(a)$ data will be given at a finite time after the initial singularity, say $t=t_0>0$, $(b)$ we will obtain the global existence, and $(c)$ the scattering kernel in this paper will differ from the one in \cite{LNT}, which was the type of soft potentials:
\[
\sigma(h,\omega) = h^{-b}\quad (1<b<2),
\]
but in this paper it will be assumed to be of the following type:
\begin{align}\label{scattering}
\sigma(h,\omega) = 
\left\{
\begin{aligned}
&h^{-b}\quad (0<b<1),\\
&h^a\quad (0\leq a<2),
\end{aligned}
\right.
\end{align}
which covers a wide range of soft and hard potentials. Unfortunately, we were not able to obtain the result for $1<b<2$, which could have lead to a global existence result with data at the initial singularity. We note that the distribution function in \cite{LNT} was assumed to be spatially homogeneous and isotropic, i.e., $f = f(t,|p|)$, but the isotropy assumption will be removed in this paper so that we only have the spatial homogeneity on the distribution function. Hence, we expect that the results of this paper can be extended to the Bianchi cases. We also expect that the idea of this paper can be used to extend \cite{LNT} to the Bianchi cases. 

The strategy of this paper is as follows. We first consider the arguments of \cite{LR}, where the global existence for general initial data was obtained in the massive case. Applying \cite{LR} to the massless case, we encounter a singularity in the collision operator (see \eqref{p^0} and \eqref{Bsoft}--\eqref{Bhard}), but it will be shown that the singularity can be controlled by using the singular weights $|p|^r$ and $|p|e^{|p|}$ (see \eqref{norm1}--\eqref{norm2}). In the case of soft potentials, we estimate the $L^1_{-2}$ norm to obtain the existence in $L^1_{-1}$. The $L^1_{-2}$ norm will be estimated by using the $L^\infty_w$ norm, and the boundedness of the $L^\infty_w$ norm will be obtained by assuming small initial data and using the expansion of the universe. In this paper, the scale factor in the FLRW metric (see \eqref{metric}) will be assumed to be given by
\begin{align}\label{R}
R = C (t+t_0)^{\frac12},
\end{align}
for some constant $C>0$, so that initial data will be given at $t=0$, and the initial singularity will be located at $t=-t_0$ (see page 4 of \cite{LNT} for more details). Similar arguments will be given in the case of hard potentials, and the following are the main results of this paper. 

\begin{thm}\label{thms}
Let $f_0$ be an initial data of the massless Boltzmann equation \eqref{Bsoft} satisfying $0\leq f_0\in L_{-2}^1(\bbr^3)\cap L^\infty_w(\bbr^3)$. Then, there exists $\varepsilon>0$ such that for any $\|f_0\|_{L^\infty_w}<\varepsilon$, the massless Boltzmann equation has a unique non-negative solution $f\in C^1([0,\infty);L^1(\bbr^3)\cap L^1_{-1}(\bbr^3))$ satisfying
\begin{align*}
&\sup_{0\leq t<\infty} \|f(t)\|_{L^\infty_w}\leq C\varepsilon.
\end{align*}
\end{thm}

\begin{thm}\label{thmh}
Let $f_0$ be an initial data of the massless Boltzmann equation \eqref{Bhard} satisfying $0\leq f_0\in L^\infty_w(\bbr^3)$. Then, there exists $\varepsilon>0$ such that for any $\|f_0\|_{L^\infty_w}<\varepsilon$, the massless Boltzmann equation has a unique non-negative solution $f\in C^1([0,\infty);L^1(\bbr^3)\cap L^1_1(\bbr^3))$ satisfying
\[
\sup_{0\leq t<\infty} \|f(t)\|_{L^\infty_w}\leq C\varepsilon. 
\]
\end{thm}

The plan of this paper is as follows. In Section 2, we introduce the massless Boltzmann equation in an FLRW background and collect some basic lemmas. In Section 3, we prove the theorems. In the case of soft potentials, we estimate the $L^1_{-2}$ norm in Proposition \ref{props}, where the estimate of the $L^\infty_w$ norm will be crucially used, and prove the global existence in $L^1(\bbr^3)\cap L^1_{-1}(\bbr^3)$ in Section 3.1. In the case of hard potentials, we need to estimate the $L^1_2$ norm, but the estimate of the $L^\infty_w$ norm will be enough to obtain the global existence in $L^1(\bbr^3)\cap L^1_1(\bbr^3)$. This will be given in Proposition \ref{proph} and Section 3.2.

%
%

\section{Preliminaries}
\subsection{Boltzmann equation}
The Boltzmann equation describes the time evolution of the distribution function $f=f(t,x,p)$, which is the density function in the phase space. Let $p^\alpha$ be the four-momentum of a particle with rest mass $m\geq 0$. By the mass shell condition:
\[
p_\alpha p^\alpha = -m^2, 
\]
we have $p^0$ as a function of $p$. In the Minkowski case we have
\[
p^0 = \sqrt{m^2 + |p|^2}.
\]
We consider only binary collisions and assume that the total energy and momentum is conserved. Let $p^\alpha$, $q^\alpha$, $p'^\alpha$, and $q'^\alpha$ denote the pre-collision and the post-collision momenta of two colliding particles. Then, we have
\[
p'^\alpha+ q'^\alpha = p^\alpha + q^\alpha.
\]
The relative momentum $h$ and the total energy $s$ are defined by
\begin{align}
h&=\sqrt{(p_\alpha - q_\alpha)(p^\alpha - q^\alpha)}=\sqrt{-2m^2 - 2p_\alpha q^\alpha},\label{h}\\
s&=-(p_\alpha +q_\alpha)(p^\alpha + q^\alpha) = 2m^2 -2p_\alpha q^\alpha,\label{s}
\end{align}
and the energy-momentum conservation shows that they are collisional invariants. In the Minkowski case the Boltzmann equation is written as follows:
\[
\partial_t f + \frac{p}{p^0}\cdot\nabla_x f = \int_{\bbr^3}\int_{\bbs^2}\frac{h\sqrt{s}}{p^0 q^0}\sigma(h,\omega)(f(p') f(q') - f(p) f(q)) d\omega dq.
\]
The quantity $\sigma$ is the scattering kernel, and the post-collision momentum $p'^\alpha$ and $q'^\alpha$ can be parametrized by
\begin{align}
p'^0 & = \frac{p^0+q^0}{2}+\frac{h(p+q)\cdot\omega}{2\sqrt{s}},\label{p'^0_ortho} \\
q'^0 & = \frac{p^0+q^0}{2}-\frac{h(p+q)\cdot\omega}{2\sqrt{s}},\label{q'^0_ortho} 
\end{align}
and
\begin{align}
p' &  = \frac{p+q}{2}+\frac{h}{2}\left(\omega +\frac{((p+q)\cdot\omega)(p+q)}{\sqrt{s}(p^0+q^0+\sqrt{s})}\right),\label{p'_ortho}\\
q' &  = \frac{p+q}{2}-\frac{h}{2}\left(\omega +\frac{((p+q)\cdot\omega)(p+q)}{\sqrt{s}(p^0+q^0+\sqrt{s})}\right).\label{q'_ortho}
\end{align}
Several different ways to parametrize the post-collision momentum are known \cite{GS,LN,LN1}, but the expressions \eqref{p'^0_ortho}--\eqref{q'_ortho} are the ones of \cite{S10,SY}. We refer to \cite{CK} for more details about the relativistic Boltzmann equation. 

In this paper we are interested in the FLRW spacetime. We will assume that the metric is given by
\begin{align}\label{metric}
g = -dt^2 + R^2((dx^1)^2 + (dx^2)^2 + (dx^3)^2)
\end{align}
and the scale factor $R$ satisfies the assumption \eqref{R}. Indices are now raised and lowered via the metric $g$ so that we have
\[
p^0 = - p_0,\quad p_i = R^2 p^i,\quad i = 1,2,3.
\]
In this paper, we will consider the spatially homogeneous case, and the distribution function will be assumed to be a function of $t$ and $p_i$. We will write
\[
p = (p_1, p_2, p_3)
\]
so that the distribution function can be written as
\[
f = f(t,p).
\]
Let us define\footnote{The notations should not be confused with the ones in \cite{LNT}, where ${\bf p}$ was used to denote a three dimensional vector, and the modulus of ${\bf p}$ was denoted by $p$.}
\[
|p | :=\sqrt{\sum_{i=1}^3(p_i)^2}.
\]
Then, we obtain from the mass shell condition
\[
p^0 = \sqrt{m^2 + R^{-2} |p|^2}. 
\]
The Boltzmann equation in an FLRW background is now written as follows:
\begin{align}\label{B}
\partial_t f = R^{-3}\int_{\bbr^3}\int_{\bbs^2}\frac{h\sqrt{s}}{p^0 q^0}\sigma(h,\omega)(f(p') f(q') - f(p) f(q)) d\omega dq,
\end{align}
where $R^{-3}$ corresponds to $(-\det g)^{-1/2}$, $dq = dq_1 dq_2 dq_3$, and $\omega = (\omega_1,\omega_2,\omega_3)$ is a unit vector such that $|\omega|=1$. The quantities \eqref{h} and \eqref{s} are now given by
\begin{align*}
h&=\sqrt{-2m^2 - 2p_\alpha q^\alpha} = \sqrt{-2m^2 +2 p^0 q^0-2R^{-2} (p\cdot q)},\\
s&= 2m^2 -2p_\alpha q^\alpha = 2m^2 + 2p^0 q^0 - 2R^{-2} (p\cdot q),
\end{align*}
where
\[
p\cdot q := p_1 q_1 + p_2 	q_2 + p_3 q_3.
\]
The expressions of post-collision momentum can be obtained by considering an orthonormal frame. For instance, we may choose $e^0 = dt(=dx^0)$ and $e^i = Rdx^i$ to obtain \eqref{p'^0_ortho}--\eqref{q'_ortho} with respect to $\{e^\alpha\}$. Hence, we obtain with respect to $\{dx^\alpha\}$ the following:
\begin{align*}
p'^0 & = \frac{p^0 + q^0}{2} +\frac{h(p+q)\cdot\omega}{2R\sqrt{s}},\\
q'^0 & = \frac{p^0 + q^0}{2} -\frac{h(p+q)\cdot\omega}{2R\sqrt{s}},
\end{align*}
and
\begin{align*}
p' & = \frac{p + q}{2} + \frac{h}{2}\left(R\omega + \frac{((p +q)\cdot\omega)(p + q) }{R\sqrt{s}(p^0 + q^0 + \sqrt{s})}\right),\\
q' & = \frac{p + q}{2} - \frac{h}{2}\left(R\omega + \frac{((p +q)\cdot\omega)(p + q) }{R\sqrt{s}(p^0 + q^0 + \sqrt{s})}\right).
\end{align*}
We note that the above expressions are the same with the expression (2.5) of \cite{L13}, where the Boltzmann equation was studied for massive particles in a given FLRW spacetime.

In this paper we will consider massless particles:
\[
m=0.
\]
By the mass shell condition we have
\begin{align}
p^0 & = R^{-1}|p|.\label{p^0}
\end{align}
Let us define
\begin{align}\label{varrho}
\varrho := \sqrt{2(|p||q|- p\cdot q)},
\end{align}
then we obtain
\begin{align}
h = \sqrt{s} = R^{-1}\varrho.\label{hvarrho}
\end{align}
Now, applying the assumption on the scattering kernel \eqref{scattering} to the equation \eqref{B} together with \eqref{p^0}--\eqref{hvarrho}, we obtain the following:
\begin{align}\label{Bsoft}
\partial_t f = R^{-3+b}\int_{\bbr^3}\int_{\bbs^2}\frac{\varrho^{2-b}}{|p||q|}(f(p') f(q') - f(p) f(q)) d\omega dq
\end{align}
in the soft potential case, and
\begin{align}\label{Bhard}
\partial_t f = R^{-3-a}\int_{\bbr^3}\int_{\bbs^2}\frac{\varrho^{2+a}}{|p||q|}(f(p') f(q') - f(p) f(q)) d\omega dq
\end{align}
in the hard potential case, where $0<b<1$ and $0\leq a<2$. The post-collision momentum can be written as
\begin{align}
|p'| & = \frac{|p| + |q|}{2} +\frac{(p+q)\cdot\omega}{2},\label{p'^0}\\
|q'| & = \frac{|p| + |q|}{2} -\frac{(p+q)\cdot\omega}{2},\label{q'^0}
\end{align}
and
\begin{align}
p' & = \frac{p + q}{2} + \frac{\varrho}{2}\left(\omega + \frac{((p +q)\cdot\omega)(p + q) }{\varrho(|p| + |q| + \varrho)}\right),\label{p'}\\
q' & = \frac{p + q}{2} - \frac{\varrho}{2}\left(\omega + \frac{((p +q)\cdot\omega)(p + q) }{\varrho(|p|+|q|+\varrho)}\right).\label{q'}
\end{align}
Moreover, by the energy conservation we have
\begin{align}\label{energy}
|p'|+|q'| = |p| + |q|,
\end{align}
and the change of variables between $(p,q)$ and $(p',q')$ is given by
\begin{align}\label{dpdq}
\frac{1}{|p||q|}dpdq = \frac{1}{|p'||q'|}dp'dq'.
\end{align}
In this paper the massless Boltzmann equation will refer to the equation \eqref{Bsoft} or \eqref{Bhard}. Note that the equation \eqref{Bsoft} is the same with the equation (21) of \cite{LNT}, where a different time coordinate was used so that the factor $R^{-3+b}$ does not appear. In the present paper we will make use of the integrability of $R^{-3+b}$ or $R^{-3-a}$, so we do not need to redefine the time coordinate. 

\subsection{Basic lemmas}
In this part we collect basic lemmas. They are almost the same with the lemmas in \cite{LNT}, but we present them for the reader's convenience. 

\begin{lemma}\label{basic}
The quantity $\varrho$ defined by \eqref{varrho} satisfies the following:
\begin{align*}
\varrho^2  &= 4|p||q|\sin^2\frac{\varphi}{2},\\
\varrho^2 & \leq 4\min\{|p| |q|, |p'||q'|\},
\end{align*}
where $\varphi$ is the angle between the three-dimensional vectors ${{p}}$ and ${{q}}$.
\end{lemma}
\begin{proof}
By the definition \eqref{varrho} we have
\begin{align*}
\varrho^2  = 2(|p|| q| - {{p}}\cdot{{q}}) = 2|p||q|(1-\cos\varphi) = 4 |p||q| \sin^2\frac{\varphi}{2}.
\end{align*}
The inequality is clear by the fact that $\varrho$ is a collisional invariant. 
\end{proof}

\begin{lemma}\label{Pov1}
The post-collision momenta satisfy for any $\delta>0$ the following:
\[
\int_{\bbs^2}\frac{1}{|p'|}d\omega = \int_{\bbs^2}\frac{1}{|q'|}d\omega \leq \frac{C}{\varrho^\delta (|p|+|q|)^{1-\delta}},
\]
where $C$ is a positive constant depending on $\delta$. 
\end{lemma}
\begin{proof}
For simplicity let us write 
\begin{align*}
\nu := |p| + |q|,\quad n:=p+q,
\end{align*}
so that we can write
\[
|p'| = \frac{\nu+n\cdot\omega}{2},\quad |q'| = \frac{\nu-n\cdot\omega}{2}.
\]
We use $\varrho^2 = \nu^2 -|n|^2$ to obtain the following:
\begin{align*}
\int_{\bbs^2}\frac{1}{|p'|}d\omega 
& = \int_{\bbs^2} \frac{2}{\nu+ n\cdot\omega} d\omega\\
& = 4\pi\int_0^\pi \frac{\sin\theta}{\nu+|n|\cos\theta} d\theta\allowdisplaybreaks\\
& = \frac{4\pi}{|n|}\ln \left(\frac{\nu+|n|}{\nu - |n|}\right)\\
& = \frac{8\pi}{|n|}\ln \left(\frac{\nu+|n|}{\varrho}\right)\allowdisplaybreaks\\
& = \frac{8\pi}{\varrho^\delta \nu^{1-\delta}}\left(1+\frac{|n|^2}{\varrho^2}\right)^{\frac{1-\delta}{2}}\frac{\ln\left(\frac{|n|}{\varrho} +\sqrt{1+\frac{|n|^2}{\varrho^2}}\right)}{\frac{|n|}{\varrho}}.
\end{align*}
Note that for any $\delta>0$ the following is bounded:
\[
(1+x^2)^{\frac{1-\delta}{2}}\frac{\ln(|x| + \sqrt{1+x^2})}{|x|}.
\]
Hence, we obtain the desired result:
\[
\int_{\bbs^2}\frac{1}{|p'|}d\omega \leq \frac{C}{\varrho^\delta \nu^{1-\delta}}.
\]
The calculation for $q'$ is the same, and this completes the proof. 
\end{proof}

\begin{lemma}\label{Pov2}
The post-collision momenta satisfy the following:
\[
\int_{\bbs^2}\frac{1}{|p'|^2}d\omega = \int_{\bbs^2}\frac{1}{|q'|^2}d\omega = \frac{16\pi}{\varrho^2}.
\]
\end{lemma}
\begin{proof}
By a direct calculation we obtain
\begin{align*}
\int_{\bbs^2}\frac{1}{|p'|^2}d\omega & = \int_{\bbs^2}\frac{4}{(\nu + n\cdot\omega)^2}d\omega\\
& = 8\pi\int_0^\pi \frac{\sin\theta}{(\nu + |n|\cos\theta)^2} d\theta\\
& = \frac{16\pi}{\varrho^2},
\end{align*}
where $\nu$ and $n$ are  the same as in the proof of the previous lemma. The calculation for $q'$ is the same, and this completes the proof. 
\end{proof}

%
%

\section{Existence of solutions}
We prove the global existence of solutions to the massless Boltzmann equation. The strategy of proving the global existence is to follow the standard arguments, for instance see \cite{CIP,MW}, but the arguments will be successfully applied to the massless case. We first consider the following modified equation:
\begin{align}
\partial_t f = Q_k(f,f),\label{Bmod}
\end{align}
where $Q_k$ is the collision operator with cutoff defined as follows: 
\begin{gather}\label{Qmods}
Q_k(f,f) := R^{-3+b}\int_{\bbr^3}\int_{\bbs^2}\mathds{1}_{\{\varrho\geq k^{-1}\}}\frac{\varrho^{2-b}}{|p||q|}(f(p') f(q') - f(p) f(q)) d\omega dq,
\end{gather}
in the soft potential case, and
\begin{gather}\label{Qmodh}
Q_k(f,f) := R^{-3-a}\int_{\bbr^3}\int_{\bbs^2}\mathds{1}_{\{\varrho\leq k\}}\frac{\varrho^{2+a}}{|p||q|}(f(p') f(q') - f(p) f(q)) d\omega dq,
\end{gather}
in the hard potential case. Notice that the kernels are bounded in both cases by Lemma \ref{basic}. The quantities $R^{-3+b}$ and $R^{-3-a}$ are decreasing, so it is easy to follow the arguments of \cite{LNT} to obtain the global existence of solutions to the modified equation. 

Next, we need to consider weighted norms in order to remove the cutoffs. Let $L^1_r(\bbr^3)$ and $L^\infty_w(\bbr^3)$ denote the spaces of functions equipped with the following norms:
\begin{align}
\|f\|_{L^1_{r}}&:= \int_{\bbr^3} |f(p) ||p|^r dp,\label{norm1}\\
\|f\|_{L^\infty_w} &: = \sup_{p\in\bbr^3} |wf(p)|,\quad w :=|p| e^{|p|}.\label{norm2}
\end{align}
Note that $\|\cdot\|_{L^1_0}$ is the usual $L^1$-norm, in which case we will write $\|\cdot\|_{L^1}$ for simplicity. 

In the following we obtain the global existence and uniform boundedness of solutions to the modified equation. We study the soft potential case in Proposition \ref{props} and the hard potential case in Proposition \ref{proph}. 

%
%

\begin{prop}\label{props}
Let $k>0$ be given. For any initial data $0\leq f_0\in L^1(\bbr^3)$ the modified equation \eqref{Bmod}  with \eqref{Qmods} has a unique non-negative solution $f\in C^1([0,\infty);L^1(\bbr^3))$. If, in addition, $f_0\in L^1_{-2}(\bbr^3)\cap L^\infty_w(\bbr^3)$, then there exists $\varepsilon>0$ such that for any $\|f_0\|_{L^\infty_w}<\varepsilon$, the corresponding solution satisfies the following:
\begin{align*}
\sup_{0\leq t<\infty}\|f (t)\|_{L^\infty_w}&\leq C\varepsilon,\\
\sup_{0\leq t<\infty} \|f(t)\|_{L^1_{-2}} &\leq C,
\end{align*}
where the constants $C$ are independent of $k$.
\end{prop}

\begin{proof}
Because of the cutoff the kernel is bounded. Hence, the existence in $L^1$ is obtained by following the same arguments as in \cite{LNT}, and we skip the proof. 

We suppose that $f_0\in L^1_{-2}(\bbr^3)\cap L^\infty_w(\bbr^3)$, and let $f\in C^1([0,\infty);L^1(\bbr^3))$ be the unique non-negative solution to the modified equation \eqref{Bmod} with \eqref{Qmods}. Multiplying the equation \eqref{Bmod} with \eqref{Qmods} by $|p|^r$, integrating it over $p$, and applying \eqref{dpdq}, we obtain the following:
\begin{align}
&\frac{d}{dt} \int_{\bbr^3} f(p)|p|^rdp\nonumber\\
&= \frac{R^{-3+b}}{2}\int_{\bbr^6}\int_{\bbs^2}\mathds{1}_{\{\varrho\geq k^{-1}\}}\frac{\varrho^{2-b}}{|p||q|}f(p)f(q)(|p'|^r + |q'|^r - |p|^r - |q|^r)d\omega dpdq.\label{Bint}
\end{align}
We immediately obtain for $r=0$,
\[
\frac{d}{dt} \int_{\bbr^3} f(p)dp =0.
\]
Since $f$ is non-negative, we obtain for all $t\geq 0$,
\begin{align}\label{f0}
\|f(t)\|_{L^1} = \|f_0\|_{L^1}.
\end{align}
In order to estimate the case $r=-2$, we first need to estimate the $L^\infty_w$ norm. Multiplying the equation \eqref{Bmod} with \eqref{Qmods} by $w$, we obtain
\begin{align}
\frac{\partial (wf)}{\partial t} & = w Q_k(f,f)\nonumber\\
& \leq R^{-3+b}\int_{\bbr^3}\int_{\bbs^2}\mathds{1}_{\{\varrho\geq k^{-1}\}}\frac{\varrho^{2-b}}{|p||q|} |p| e^{|p|} f(p') f(q') d\omega dq\allowdisplaybreaks\nonumber\\
&\leq R^{-3+b}\|f\|_{L^\infty_w}^2\int_{\bbr^3}\int_{\bbs^2} \frac{\varrho^{2-b}}{|q|} e^{|p|} \frac{1}{|p'|} e^{-|p'|} \frac{1}{|q'|} e^{-|q'|} d\omega dq\nonumber\\
& = R^{-3+b}\|f\|_{L^\infty_w}^2\int_{\bbr^3}\int_{\bbs^2} \frac{\varrho^{2-b}}{|q|} e^{-|q|}\frac{1}{|p'||q'|} d\omega dq,\label{westimate}
\end{align}
where we used \eqref{energy}. The integration on $\bbs^2$ is as follows:
\begin{align*}
\int_{\bbs^2}\frac{d\omega }{|p'||q'|}
= \int_{\bbs^2}\frac{4 d\omega }{\nu^2 - (n\cdot \omega)^2}
= \int_0^{\pi}\frac{8\pi \sin\theta d\theta }{\nu^2 - |n|^2\cos^2 \theta},
\end{align*}
where $\nu$ and $n$ are the same as in the proof of Lemma \ref{Pov1}. Then, we have
\begin{align*}
\int_0^{\pi}\frac{8\pi \sin\theta d\theta }{\nu^2 - |n|^2\cos^2 \theta}
& = \frac{8\pi}{\nu |n|}\ln\left(\frac{\nu + |n|}{\nu - |n|}\right)
 = \frac{16\pi}{\nu |n|}\ln\left(\frac{\nu + |n|}{\varrho}\right).
\end{align*}
The last quantity can be estimated as in the proof of Lemma \ref{Pov1}: for any $\delta>0$ we have
\[
\frac{16\pi}{\nu |n|}\ln\left(\frac{\nu + |n|}{\varrho}\right)
\leq \frac{C}{\varrho^\delta \nu^{2-\delta}}.
\]
Since $0< b<1$, we can choose $\delta = 2-b$ to obtain
\begin{align}\label{pinnu}
\int_{\bbs^2}\frac{d\omega }{|p'||q'|} \leq \frac{C}{\varrho^{2-b}\nu^b}\leq \frac{C}{\varrho^{2-b}|q|^b},
\end{align}
where we used the fact that $b> 0$ in the last inequality. Now, we have
\begin{align*}
\frac{\partial (wf)}{\partial t} &\leq C R^{-3+b}\|f\|^2_{L^\infty_w}\int_{\bbr^3} \frac{1}{|q|^{1+b}} e^{-|q|} dq\\
&\leq C R^{-3+b}\|f\|^2_{L^\infty_w},
\end{align*}
where the integral above is finite since $0< b<1$. Then, we obtain
\[
\frac{d}{d t}\|f\|_{L^\infty_w} \leq C R^{-3+b}\|f\|^2_{L^\infty_w}.
\]
Since $R^{-3+b}$ is integrable, we conclude that there exists $\varepsilon>0$ such that if $\|f_0\|_{L^\infty_w}\leq\varepsilon$, then
\begin{align}\label{finfty}
\sup_{0\leq t<\infty}\|f (t)\|_{L^\infty_w}\leq C\varepsilon.
\end{align}
We now estimate the expression \eqref{Bint} in the case $r=-2$ as follows:
\begin{align*}
&\frac{d}{dt} \int_{\bbr^3} f(p)\frac{1}{|p|^2}dp\\
&\leq \frac{R^{-3+b}}{2}\int_{\bbr^6}\int_{\bbs^2}\mathds{1}_{\{\varrho\geq k^{-1}\}}\frac{\varrho^{2-b}}{|p||q|}f(p)f(q)\left(\frac{1}{|p'|^2} + \frac{1}{|q'|^2}\right)d\omega dpdq\allowdisplaybreaks\\
&\leq CR^{-3+b} \int_{\bbr^6}\frac{\varrho^{-b}}{|p||q|}f(p)f(q) dpdq\\
&\leq CR^{-3+b} \int_{\bbr^6}\frac{f(p)f(q)}{|p|^{1+\frac{b}{2}}|q|^{1+\frac{b}{2}}\sin^b(\varphi/2)} dpdq,
\end{align*}
where we used Lemma \ref{Pov2} and Lemma \ref{basic}. Let us consider the integration over $q$ on the right hand side. We use \eqref{finfty} to obtain the following:
\begin{align*}
\int_{\bbr^3}\frac{f(q)}{|q|^{1+\frac{b}{2}}\sin^b(\varphi/2)}dq 
& \leq  C\varepsilon\int_{\bbr^3} \frac{e^{-|q|}}{|q|^{2+\frac{b}{2}}\sin^b(\varphi/2)}dq\nonumber\\
& \leq C\varepsilon\int_0^\infty \int_0^\pi \frac{e^{-|q|}\sin\varphi}{|q|^{\frac{b}{2}}\sin^b(\varphi/2)} d\varphi d|q|\nonumber\\
& \leq C\varepsilon\int_0^\infty  \frac{e^{-|q|}}{|q|^{\frac{b}{2}}} d|q|,
\end{align*}
where the last integral is finite, since $0<b<1$. Hence, we obtain
\begin{align*}
\frac{d}{dt} \int_{\bbr^3} f(p)\frac{1}{|p|^2}dp 
&\leq C\varepsilon R^{-3+b} \int_{\bbr^3}\frac{f(p)}{|p|^{1+\frac{b}{2}}} dp\\
&\leq C\varepsilon R^{-3+b} \left(\|f\|_{L^1} + \|f\|_{L^1_{-2}}\right).
\end{align*}
Since $R^{-3+b}$ is integrable, we obtain the desired result by Gr{\"o}nwall's inequality together with \eqref{f0}. 
\end{proof}

\begin{remark}\label{rmk1}
In \cite{LNT}, the authors studied the massless Boltzmann equation in a different situation, but it was also necessary to estimate the $L^1_{-2}$ norm. In Proposition \ref{props}, we could estimate the $L^1_{-2}$ norm by using the $L^\infty_w$ norm. On the other hand, in \cite{LNT}, the distribution function $f$ was assumed to be isotropic, i.e., $f(p) = f(|p|)$, so that the $L^1_{-2}$ norm could be estimated without the $L^\infty_w$ norm. One might expect that by using the $L^\infty_w$ norm it should be possible to extend the result of \cite{LNT} to the Bianchi case.  
\end{remark}

\begin{remark}\label{rmk2}
Note that the first result of Proposition \ref{props} shows that
\[
\sup_{0\leq t<\infty}\|f(t)\|_{L^1_r}\leq C\varepsilon,
\]
for any $-2< r \leq 0$. 
\end{remark}

%
%

\begin{prop}\label{proph}
Let $k>0$ be given. For any initial data $0\leq f_0\in L^1(\bbr^3)$ the modified equation \eqref{Bmod}  with \eqref{Qmodh} has a unique non-negative solution $f\in C^1([0,\infty);L^1(\bbr^3))$. If, in addition, $f_0\in L^\infty_w(\bbr^3)$, then there exists $\varepsilon>0$ such that for any $\|f_0\|_{L^\infty_w}<\varepsilon$, the corresponding solution satisfies the following:
\[
\sup_{0\leq t<\infty}\|f(t)\|_{L^\infty_w}\leq C\varepsilon,
\]
where the constant $C$ is independent of $k$.
\end{prop}

\begin{proof}
As in the Proposition \ref{props} one can easily obtain the existence in $L^1$. We now assume that $f_0 \in L^\infty_w(\bbr^3)$. Multiplying the equation \eqref{Bmod} with \eqref{Qmodh} by $w$, we obtain
\begin{align*}
\frac{\partial (wf)}{\partial t} & = w Q_k(f,f)\\
& \leq R^{-3-a}\int_{\bbr^3}\int_{\bbs^2}\mathds{1}_{\{\varrho\leq k\}}\frac{\varrho^{2+a}}{|p||q|} |p| e^{|p|} f(p') f(q') d\omega dq\\
&\leq R^{-3-a}\|f\|_{L^\infty_w}^2\int_{\bbr^3}\int_{\bbs^2} \frac{\varrho^{2+a}}{|q|} e^{|p|} \frac{1}{|p'|} e^{-|p'|} \frac{1}{|q'|} e^{-|q'|} d\omega dq\\
& = R^{-3-a}\|f\|_{L^\infty_w}^2\int_{\bbr^3}\int_{\bbs^2} \frac{\varrho^{2+a}}{|q|} e^{-|q|}\frac{1}{|p'||q'|} d\omega dq,
\end{align*}
where we used \eqref{energy}. By the same calculation as in the Proposition \ref{props} we obtain for any $\delta>0$, 
\begin{align}\label{p'q'estimate}
\int_{\bbs^2}\frac{d\omega }{|p'||q'|}
\leq \frac{C}{\varrho^\delta \nu^{2-\delta}}.
\end{align}
Here, if we choose $\delta = 2+a$ as in the Proposition \ref{props}, then the estimate \eqref{p'q'estimate} shows that
\[
\int_{\bbs^2}\frac{d\omega }{|p'||q'|}
\leq \frac{C\nu^a}{\varrho^{2+a}},
\]
so that the dependence on $p$ remains in $\nu$ (see \eqref{pinnu} in the soft potential case). Instead, we choose $0<\delta<2$ such that
\[
a+\delta<2. 
\]
Then, we apply Young's inequality as follows:
\begin{align}\label{Young}
\nu  \geq c|p|^{\frac{2+a-\delta}{2(2-\delta)}}|q|^{\frac{2-a-\delta}{2(2-\delta)}},
\end{align}
where $c$ is a positive constant depending on $a$ and $\delta$. Hence, the $L^\infty_w$ norm can be estimated together with \eqref{p'q'estimate} and \eqref{Young} as follows:
\begin{align*}
\frac{\partial (wf)}{\partial t}
& \leq CR^{-3-a}\|f\|_{L^\infty_w}^2\int_{\bbr^3} \frac{\varrho^{2+a-\delta}}{|q|} e^{-|q|}\frac{1}{ |p|^{\frac{2+a-\delta}{2}}|q|^{\frac{2-a-\delta}{2}}} dq\\
&\leq CR^{-3-a}\|f\|_{L^\infty_w}^2\int_{\bbr^3}|q|^{a-1}e^{-|q|} dq\\
&\leq CR^{-3-a}\|f\|_{L^\infty_w}^2,
\end{align*}
where we used Lemma \ref{basic}, and the last integral is finite since $0\leq a<2$. Therefore, we conclude that there exists $\varepsilon>0$ such that if $\|f_0\|_{L^\infty_w}\leq\varepsilon$, then
\[
\sup_{0\leq t<\infty}\|f (t)\|_{L^\infty_w}\leq C\varepsilon,
\]
which completes the proof.
\end{proof}

\begin{remark}\label{rmk3}
Note that the above result shows that for any $r\geq 0$,
\[
\sup_{0\leq t<\infty}\|f(t)\|_{L^1_r}\leq C\varepsilon.
\]
On the other hand, in a similar way to the soft potential case we obtain
\begin{align*}
&\frac{d}{dt} \int_{\bbr^3} f(p)|p|^rdp\nonumber\\
&= \frac{R^{-3-a}}{2}\int_{\bbr^6}\int_{\bbs^2}\mathds{1}_{\{\varrho\leq k\}}\frac{\varrho^{2+a}}{|p||q|}f(p)f(q)(|p'|^r + |q'|^r - |p|^r - |q|^r)d\omega dpdq.
\end{align*}
Hence, we obtain $\|f(t)\|_{L^1_r} = \|f_0\|_{L^1_r}$ for all $t\geq 0$ in the cases $r=0,1$.
\end{remark}

%
%

\subsection{Proof of Theorem 1}
We are now ready to remove the cutoff. We first consider Theorem \ref{thms} for the soft potential case. Note that $f_0 \in L^\infty_w(\bbr^3)$ implies $f_0 \in L^1(\bbr^3)$. Hence, we can apply Proposition \ref{props} to obtain a sequence $\{f_k\}_{k=1}^\infty$, which are the solutions to the modified equation \eqref{Bmod} with \eqref{Qmods}:
\[
\partial_t f_k =  Q_k(f_k,f_k),\quad f_k(0) = f_0\geq 0.
\]
Below, we will show that the sequence $\{f_k\}_{k=1}^\infty$ converges in $L^1(\bbr^3)\cap L^1_{-1}(\bbr^3)$. For $m<n$, we have
\begin{align*}
\partial_t f_m -\partial_t f_n & = Q_m(f_m,f_m) - Q_m(f_n,f_n) + Q_m(f_n,f_n) - Q_n(f_n,f_n),
\end{align*}
where
\begin{align*}
&Q_m(f_m,f_m) - Q_m(f_n,f_n)\\
& = \frac{R^{-3+b}}{2}\int_{\bbr^3}\int_{\bbs^2} \mathds{1}_{\{\varrho\geq m^{-1}\}}\frac{\varrho^{2-b}}{|p||q|}\allowdisplaybreaks\\
&\qquad\times\Big\{ (f_m+f_n)(p')(f_m - f_n)(q')+(f_m +f_n)(q')(f_m - f_n)(p')\\
&\qquad\qquad - (f_m+f_n)(p)(f_m - f_n)(q) -(f_m +f_n)(q)(f_m - f_n)(p)\Big\} d\omega dq,
\end{align*}
and
\begin{align*}
&Q_m(f_n,f_n) - Q_n(f_n,f_n)\\
& = - R^{-3+b}\int_{\bbr^3}\int_{\bbs^2} \mathds{1}_{\{n^{-1}\leq\varrho\leq m^{-1}\}}\frac{\varrho^{2-b}}{|p||q|}(f_n(p')f_n(q')-f_n(p)f_n(q))d\omega dq.
\end{align*}
Multiplying the above equation by $\mbox{sgn}(f_m - f_n)(p)$ we obtain
\begin{align*}
&\partial_t|f_m - f_n|(p)\nonumber\\
&\leq \frac{R^{-3+b}}{2}\int_{\bbr^3}\int_{\bbs^2} \mathds{1}_{\{\varrho\geq m^{-1}\}}\frac{\varrho^{2-b}}{|p||q|}\allowdisplaybreaks\\
&\qquad\times \Big\{ (f_m+f_n)(p')|f_m - f_n|(q')+(f_m +f_n)(q')|f_m - f_n|(p')\\
&\qquad\qquad + (f_m+f_n)(p)|f_m - f_n|(q) -(f_m +f_n)(q)|f_m - f_n|(p)\Big\} d\omega dq\allowdisplaybreaks\\
&\quad + R^{-3+b}\int_{\bbr^3}\int_{\bbs^2} \mathds{1}_{\{n^{-1}\leq \varrho\leq m^{-1}\}}\frac{\varrho^{2-b}}{|p||q|}(f_n(p')f_n(q')+f_n(p)f_n(q))d\omega dq,
\end{align*}
where we used the fact that the solutions are non-negative. Then, multiplying the above by $|p|^r$ and integrating it over $\bbr^3$, we obtain the following:
\begin{align*}
&\frac{d}{dt}\|f_m - f_n\|_{L^1_r}\nonumber\\
& \leq \frac{R^{-3+b}}{2}\int_{\bbr^6}\int_{\bbs^2} \mathds{1}_{\{\varrho\geq m^{-1}\}} \frac{\varrho^{2-b}}{|p||q|}\allowdisplaybreaks\nonumber\\
&\qquad \times \Big\{ (f_m+f_n)(p')|f_m - f_n|(q')+(f_m +f_n)(q')|f_m - f_n|(p')\nonumber\\
& \qquad\qquad + (f_m+f_n)(p)|f_m - f_n|(q) -(f_m +f_n)(q)|f_m - f_n|(p) \Big\} |p|^r d\omega dq dp\allowdisplaybreaks\nonumber\\
&\quad + R^{-3+b}\int_{\bbr^6}\int_{\bbs^2} \mathds{1}_{\{n^{-1}\leq \varrho\leq m^{-1}\}}\frac{\varrho^{2-b}}{|p||q|}\Big\{f_n(p')f_n(q')+f_n(p)f_n(q)\Big\}|p|^rd\omega dqdp\allowdisplaybreaks\nonumber\\
& = \frac{R^{-3+b}}{2}\int_{\bbr^6}\int_{\bbs^2} \mathds{1}_{\{\varrho\geq m^{-1}\}} \frac{\varrho^{2-b}}{|p||q|}\nonumber\\
&\qquad\times(f_m+f_n)(p)|f_m - f_n|(q)\Big\{|p'|^r+|q'|^r +|p|^r - |q|^r\Big\} d\omega dqdp\allowdisplaybreaks\nonumber\\
&\quad + R^{-3+b}\int_{\bbr^6}\int_{\bbs^2} \mathds{1}_{\{n^{-1}\leq \varrho\leq m^{-1}\}}\frac{\varrho^{2-b}}{|p||q|}f_n(p)f_n(q)\Big\{|p'|^r+|p|^r\Big\}d\omega dqdp.
\end{align*}
For $r=0$ we have
\begin{align*}
&\frac{d}{dt}\|f_m - f_n\|_{L^1}\leq I_1 + I_2,
\end{align*}
where
\begin{align*}
I_1 & = R^{-3+b}\int_{\bbr^6}\int_{\bbs^2}  \frac{\varrho^{2-b}}{|p||q|}(f_m+f_n)(p)|f_m - f_n|(q) d\omega dqdp,\\
I_2 &= 2R^{-3+b}\int_{\bbr^6}\int_{\bbs^2} \mathds{1}_{\{\varrho\leq m^{-1}\}}\frac{\varrho^{2-b}}{|p||q|}f_n(p)f_n(q)d\omega dqdp.
\end{align*}
The integrals $I_1$ and $I_2$ are estimated as follows:
\begin{align}
I_1&\leq CR^{-3+b}\int_{\bbr^6} \frac{1}{|p|^{\frac{b}{2}}|q|^{\frac{b}{2}}} (f_m+f_n)(p)|f_m - f_n|(q)dq dp\nonumber\\
&\leq CR^{-3+b}\sup_k \|f_k\|_{L^1_{-b/2}}\|f_m - f_n\|_{L^1_{-b/2}},\label{sI_1}
\end{align}
and
\begin{align}
I_2&\leq CR^{-3+b}m^{-2+b}\int_{\bbr^6} \frac{1}{|p||q|}f_n(p) f_n(q) dq dp\nonumber\\
& \leq C R^{-3+b}m^{-2+b}\sup_k \|f_k\|^2_{L^1_{-1}}. \label{sI_2}
\end{align}
For $r=-1$, we have
\begin{align*}
\frac{d}{dt}\|f_m - f_n\|_{L^1_{-1}}\leq J_1 + J_2 + J_3 + J_4 + J_5,
\end{align*}
where
\begin{align*}
J_1 & = \frac{R^{-3+b}}{2}\int_{\bbr^6}\int_{\bbs^2}  \frac{\varrho^{2-b}}{|p||q|}(f_m+f_n)(p)|f_m - f_n|(q)\frac{1}{|p'|}d\omega dqdp,\\
J_2 & = \frac{R^{-3+b}}{2}\int_{\bbr^6}\int_{\bbs^2}  \frac{\varrho^{2-b}}{|p||q|}(f_m+f_n)(p)|f_m - f_n|(q)\frac{1}{|q'|}d\omega dqdp,\\
J_3 & = \frac{R^{-3+b}}{2}\int_{\bbr^6}\int_{\bbs^2}  \frac{\varrho^{2-b}}{|p||q|}(f_m+f_n)(p)|f_m - f_n|(q)\frac{1}{|p|}d\omega dqdp,\allowdisplaybreaks\\
J_4 & = R^{-3+b}\int_{\bbr^6}\int_{\bbs^2} \mathds{1}_{\{\varrho\leq m^{-1}\}}\frac{\varrho^{2-b}}{|p||q|}f_n(p)f_n(q)\frac{1}{|p'|}d\omega dqdp,\\
J_5 & = R^{-3+b}\int_{\bbr^6}\int_{\bbs^2} \mathds{1}_{\{\varrho\leq m^{-1}\}}\frac{\varrho^{2-b}}{|p||q|}f_n(p)f_n(q)\frac{1}{|p|}d\omega dqdp.
\end{align*}
We use Lemma \ref{Pov1} to estimate $J_1$ as follows:
\begin{align}
J_1 & = \frac{R^{-3+b}}{2}\int_{\bbr^6}\int_{\bbs^2}  \frac{\varrho^{2-b}}{|p||q|}(f_m+f_n)(p)|f_m - f_n|(q)\frac{1}{|p'|}d\omega dqdp\nonumber\\
& \leq CR^{-3+b}\int_{\bbr^6} \frac{\varrho^{1-b}}{|p||q|}(f_m+f_n)(p)|f_m - f_n|(q) dqdp\allowdisplaybreaks\nonumber\\
& \leq CR^{-3+b}\int_{\bbr^6} \frac{1}{|p|^{\frac{1+b}{2}}|q|^{\frac{1+b}{2}}}(f_m+f_n)(p)|f_m - f_n|(q) dqdp\nonumber\\
& \leq CR^{-3+b}\sup_k \|f_k\|_{L^1_{-(1+b)/2}}\|f_m - f_n\|_{L^1_{-(1+b)/2}}.\label{sJ_1}
\end{align}
The estimate of $J_2$ is exactly the same with that of $J_1$:
\begin{align}
J_2&\leq CR^{-3+b}\sup_k \|f_k\|_{L^1_{-(1+b)/2}}\|f_m - f_n\|_{L^1_{-(1+b)/2}}.\label{sJ_2}
\end{align}
The integral $J_3$ is estimated as follows:
\begin{align}
J_3 & = \frac{R^{-3+b}}{2}\int_{\bbr^6}\int_{\bbs^2}  \frac{\varrho^{2-b}}{|p||q|}(f_m+f_n)(p)|f_m - f_n|(q)\frac{1}{|p|}d\omega dqdp\nonumber\\
&\leq CR^{-3+b}\int_{\bbr^3} (f_m + f_n)(p)\frac{1}{|p|^{1+\frac{b}{2}}}dp \int_{\bbr^3} |f_m - f_n|(q)\frac{1}{|q|^{\frac{b}{2}}} dq\nonumber\\
&\leq C R^{-3+b}\sup_k \|f_k\|_{L^1_{-1-b/2}} \|f_m - f_n\|_{L^1_{-b/2}}.\label{sJ_3}
\end{align}
For the integral $J_4$, since $0< b<1$, we have the following:
\begin{align}
J_4 & = R^{-3+b}\int_{\bbr^6}\int_{\bbs^2} \mathds{1}_{\{\varrho\leq m^{-1}\}}\frac{\varrho^{2-b}}{|p||q|}f_n(p)f_n(q)\frac{1}{|p'|}d\omega dqdp\nonumber\\
&\leq CR^{-3+b}\int_{\bbr^6} \mathds{1}_{\{\varrho\leq m^{-1}\}}\frac{\varrho^{1-b}}{|p||q|}f_n(p)f_n(q) dqdp\nonumber\allowdisplaybreaks\\
&\leq C R^{-3+b} m^{-1+b}\int_{\bbr^6} \frac{1}{|p||q|}f_n(p)f_n(q) dqdp\nonumber\\
&\leq C R^{-3+b} m^{-1+b}\sup_k \|f_k\|^2_{L^1_{-1}}.\label{sJ_4}
\end{align}
Similarly, $J_5$ is estimate as follows:
\begin{align}
J_5 & = R^{-3+b}\int_{\bbr^6}\int_{\bbs^2} \mathds{1}_{\{\varrho\leq m^{-1}\}}\frac{\varrho^{2-b}}{|p||q|}f_n(p)f_n(q)\frac{1}{|p|}d\omega dqdp\nonumber\\
&\leq CR^{-3+b} m^{-2+b}\int_{\bbr^3} f_n(p)\frac{1}{|p|^2}dp \int_{\bbr^3} f_n(q)\frac{1}{|q|} dq\nonumber\\
&\leq CR^{-3+b}  m^{-2+b}\sup_k \|f_k\|_{L^1_{-2}}\sup_k \|f_k\|_{L^1_{-1}}.\label{sJ_5}
\end{align}
We now apply Proposition \ref{props} with Remark \ref{rmk2} to obtain from \eqref{sI_1}--\eqref{sJ_5} that
\begin{multline*}
\frac{d}{dt}\left(\|f_m - f_n\|_{L^1}+\|f_m - f_n\|_{L^1_{-1}}\right)\\
\leq C\varepsilon R^{-3+b}\left( m^{-1+b} + \|f_m - f_n\|_{L^1} + \|f_m - f_n\|_{L^1_{-1}}\right).
\end{multline*}
Since $R^{-3+b}$ is integrable and $f_m(0) = f_n(0) = f_0$, we obtain
\[
\|f_m - f_n\|_{L^1}+\|f_m - f_n\|_{L^1_{-1}} \leq C m^{-1+b},
\]
which shows that the sequence converges in $L^1(\bbr^3)\cap L^1_{-1}(\bbr^3)$ as $m\to\infty$. Hence, we obtain the existence part of Theorem \ref{thms}. To obtain the boundedness of $\|f\|_{L^\infty_w}$ we multiply the equation \eqref{Bsoft} by $w$ and follow the proof of Proposition \ref{props}. Note that the estimate \eqref{westimate} still holds for the original equation \eqref{Bsoft} without the cutoff. Therefore, we obtain again the estimate \eqref{finfty}, and this completes the proof of Theorem \ref{thms}.

%
%

\subsection{Proof of Theorem 2}
We now consider the hard potential case. The strategy is basically the same as in the soft potential case, but we consider the $L^1_1$ norm instead. By the same arguments we obtain for $m<n$:
\begin{align*}
&\frac{d}{dt}\|f_m - f_n\|_{L^1_r}\\
& \leq \frac{R^{-3-a}}{2}\int_{\bbr^6}\int_{\bbs^2} \mathds{1}_{\{\varrho\leq m\}} \frac{\varrho^{2+a}}{|p||q|}\allowdisplaybreaks\\
&\qquad\times(f_m+f_n)(p)|f_m - f_n|(q)\Big\{|p'|^r+|q'|^r +|p|^r - |q|^r\Big\} d\omega dqdp\\
&\quad + R^{-3-a}\int_{\bbr^6}\int_{\bbs^2} \mathds{1}_{\{m\leq \varrho\leq n\}}\frac{\varrho^{2+a}}{|p||q|}f_n(p)f_n(q)\Big\{|p'|^r+|p|^r\Big\}d\omega dqdp,
\end{align*}
where $f_m$ and $f_n$ are the solutions of \eqref{Bmod} with \eqref{Qmodh}. Then, for $r=0$ we have
\begin{align*}
&\frac{d}{dt}\|f_m - f_n\|_{L^1}\leq I_1 + I_2,
\end{align*}
where
\begin{align*}
I_1 & = R^{-3-a}\int_{\bbr^6}\int_{\bbs^2}  \frac{\varrho^{2+a}}{|p||q|}(f_m+f_n)(p)|f_m - f_n|(q) d\omega dqdp,\\
I_2 &= 2R^{-3-a}\int_{\bbr^6}\int_{\bbs^2} \mathds{1}_{\{\varrho\geq m\}}\frac{\varrho^{2+a}}{|p||q|}f_n(p)f_n(q)d\omega dqdp.
\end{align*}
The integrals $I_1$ and $I_2$ are estimated as follows:
\begin{align}
I_1&\leq CR^{-3-a}\int_{\bbr^6} |p|^{\frac{a}{2}}|q|^{\frac{a}{2}} (f_m+f_n)(p)|f_m - f_n|(q)dq dp\nonumber\\
&\leq CR^{-3-a} \sup_k \|f_k\|_{L^1_{a/2}}\|f_m - f_n\|_{L^1_{a/2}},\label{hI_1}
\end{align}
and since $0\leq a<2$,
\begin{align}
I_2 &\leq CR^{-3-a}m^{-2+a}\int_{\bbr^6}\int_{\bbs^2} \mathds{1}_{\{\varrho\geq m\}}\frac{\varrho^{4}}{|p||q|}f_n(p)f_n(q)d\omega dqdp\nonumber\\
& \leq C R^{-3-a}m^{-2+a}\sup_k \|f_k\|_{L^1_1}^2. \label{hI_2}
\end{align}
For $r=1$, we have
\begin{align*}
\frac{d}{dt}\|f_m - f_n\|_{L^1_{1}}\leq J_1 + J_2,
\end{align*}
where
\begin{align*}
J_1 & = R^{-3-a}\int_{\bbr^6}\int_{\bbs^2}  \frac{\varrho^{2+a}}{|p||q|}(f_m+f_n)(p)|f_m - f_n|(q)|p|d\omega dqdp,\\
J_2 & = R^{-3-a}\int_{\bbr^6}\int_{\bbs^2} \mathds{1}_{\{\varrho\geq m\}} \frac{\varrho^{2+a}}{|p||q|}f_n(p)f_n(q)\Big\{|p'|+|p|\Big\}d\omega dqdp,
\end{align*}
where we used \eqref{energy} for $J_1$. We estimate $J_1$ as follows:
\begin{align}
J_1 & \leq CR^{-3-a}\int_{\bbr^6} |p|^{1+\frac{a}{2}}(f_m+f_n)(p)|q|^{\frac{a}{2}}|f_m - f_n|(q) dqdp\allowdisplaybreaks\nonumber\\
& \leq CR^{-3-a} \sup_k \|f_k\|_{L^1_{1+a/2}} \|f_m - f_n\|_{L^1_{a/2}}.\label{hJ_1}
\end{align}
By \eqref{energy} and the symmetry, $J_2$ can be estimated as
\begin{align}
J_2 & \leq CR^{-3-a}\int_{\bbr^6}\int_{\bbs^2} \mathds{1}_{\{\varrho\geq m\}} \frac{\varrho^{2+a}}{|p||q|}f_n(p)f_n(q)|p|d\omega dqdp\nonumber\\
& \leq CR^{-3-a}m^{-2+a}\int_{\bbr^6} \frac{\varrho^{4}}{|p||q|}f_n(p)f_n(q)|p| dqdp\nonumber\\
& \leq CR^{-3-a}m^{-2+a}\sup_k\|f_k\|_{L^1_2}\sup_k \|f_k\|_{L^1_1}\label{hJ_2}.
\end{align}
Applying Proposition \ref{proph} and the estimates of Remark \ref{rmk3}, we obtain from \eqref{hI_1}--\eqref{hJ_2} that
\begin{multline*}
\frac{d}{dt}\left(\|f_m - f_n\|_{L^1}+\|f_m - f_n\|_{L^1_{1}}\right)\\
\leq C\varepsilon R^{-3-a}\left( m^{-2+a} + \|f_m - f_n\|_{L^1} + \|f_m - f_n\|_{L^1_{1}}\right).
\end{multline*}
Since $0\leq a<2$, one can prove the existence of solutions in $L^1(\bbr^3)\cap L^1_1(\bbr^3)$ as in the proof of Theorem \ref{thms}. The boundedness of $\|f\|_{L^\infty_w}$ can be obtained by the same calculation as in Proposition \ref{proph}. This completes the proof of Theorem \ref{thmh}.

\section*{Acknowledgements}
The author thanks Ernesto Nungesser and Paul Tod for interesting comments in connection with our previous work. This work was supported by the Basic Science Research Program through the National Research Foundation of Korea (NRF) funded by the Ministry of Science, ICT \& Future Planning (NRF-2018R1A1A1A05078275).

\bibliographystyle{abbrv}

\begin{thebibliography}{10}

\bibitem{A96}
Andr{\'e}asson, H.
Regularity of the gain term and strong $L^1$ convergence to equilibrium for the relativistic Boltzmann equation.
{\it SIAM J. Math. Anal.} 27 (1996), no. 5, 1386--1405.

\bibitem{ACI}
Andr{\'e}asson, H., Calogero, S., and Illner, R.
On blowup for gain-term-only classical and relativistic Boltzmann equations.
{\it Math. Methods Appl. Sci.} 27 (2004), no. 18, 2231--2240. 

\bibitem{A} 
Anguige, K.
The Cauchy problem for the inhomogeneous conformal Einstein-Vlasov equations.
{\it Ann. Phys.} 282 (2000), 395--419. 

\bibitem{AT99a} 
Anguige, K. and Tod, K. P.
Isotropic cosmological singularities. I. Polytropic perfect-fluid space-times.
{\it Ann. Phys.} 276 (1999), 257--293.

\bibitem{AT99}
Anguige, K. and Tod, K. P.
Isotropic cosmological singularities. II. The Einstein-Vlasov system.
{\it Ann. Phys.}, 276 (1999), 294--320.
 



\bibitem{B73}
Bancel, D.
Probl{\`e}me de Cauchy pour l'{\'e}quation de Boltzmann en relativit{\'e} g{\'e}n{\'e}rale. (French)
{\it Ann. Inst. H. Poincar{\'e} Sect. A (N.S.)} 18 (1973), 263--284.

\bibitem{BCB}
Bancel, D. and Choquet-Bruhat, Y.
Existence, uniqueness, and local stability for the Einstein-Maxwell-Boltzman system.
{\it Comm. Math. Phys.} 33 (1973), 83--€"96.


\bibitem{baz1} 
Bazow, D., Denicol, G. S., Heinz, U., Martinez, M., and Noronha, J. 
\newblock{Analytic solution of the Boltzmann equation in an expanding system.}
\newblock{\it Phys. Rev. Lett.} 116 (2016), 022301.

\bibitem{baz2} 
Bazow, D., Denicol, G. S., Heinz, U., Martinez, M., and Noronha, J. 
Nonlinear dynamics from the relativistic Boltzmann equation in the Friedmann-Lema{\^i}tre-Robertson-Walker spacetime.
{\it Phys. Rev. D} 94 (2016), 125006.

\bibitem{B67}
Bichteler, K.
On the Cauchy problem of the relativistic Boltzmann equation.
{\it Comm. Math. Phys.} 4 (1967), 352--364.

\bibitem{CIP}
Cercignani, C., Illner, R., and Pulvirenti, M.
The mathematical theory of dilute gases.
Applied Mathematical Sciences, 106. {\it Springer-Verlag, New York,} 1994.

\bibitem{CK}
Cercignani, C. and Kremer, G. M.
The relativistic Boltzmann equation: theory and applications.
Progress in Mathematical Physics, 22. {\it Birkh{\"a}user Verlag, Basel}, 2002.



\bibitem{DEJ}
Dudy\'nski, M. T. and Ekiel-Je\.zewska, M. L.
On the linearized relativistic Boltzmann equation. I. Existence of solutions.
{\it Comm. Math. Phys.} 115 (1985), no. 4, 607--629.

\bibitem{DEJ92}
Dudy\'nski, M. T. and Ekiel-Je\.zewska, M. L.
Global existence proof for relativistic Boltzmann equation.
{\it J. Statist. Phys.} 66 (1992), no. 3-4, 991--1001.



\bibitem{G}
Glassey, R. T.
The Cauchy problem in kinetic theory. 
{\it Society for Industrial and Applied Mathematics (SIAM), Philadelphia, PA}, 1996. 

\bibitem{G06}
Glassey, R. T.
Global solutions to the Cauchy problem for the relativistic Boltzmann equation with near-vacuum data.
{\it Comm. Math. Phys.} 264 (2006), no. 3, 705--724.
 
\bibitem{GS}
Glassey, R. T. and Strauss, W. A. 
Asymptotic stability of the relativistic Maxwellian.
{\it Publ. Res. Inst. Math. Sci.} 29 (1993), no. 2, 301--347. 

\bibitem{GS95}
Glassey, R. T. and Strauss, W. A.
Asymptotic stability of the relativistic Maxwellian via fourteen moments.
{\it Transport Theory Statist. Phys.} 24 (1995), no. 4-5, 657--678.

\bibitem{GS12}
Guo, Y. and Strain, R. M.
Momentum regularity and stability of the relativistic Vlasov-Maxwell-Boltzmann system.
{\it Comm. Math. Phys.} 310 (2012), no. 3, 649--673.

\bibitem{JY}
Jang, J. W. and Yun, S.-B.
Gain of regularity for the relativistic collision operator.
{\it Appl. Math. Lett.} 90 (2019), 162--169.



\bibitem{L13}
Lee, H.
Asymptotic behaviour of the relativistic Boltzmann equation in the Robertson-Walker spacetime.
{\it J. Differential Equations} 255 (2013), no. 11, 4267--4288.

\bibitem{LN}
Lee, H. and Nungesser, E.
Future global existence and asymptotic behaviour of solutions to the Einstein-Boltzmann system with Bianchi I symmetry.
{\it J. Differential Equations} 262 (2017), no.11, 5425--5467.

\bibitem{LN1}
Lee, H. and Nungesser, E.
Late-time behaviour of Israel particles in a FLRW spacetime with $\Lambda>0$.
{\it J. Differential Equations} 263 (2017), no. 1, 841--862. 

\bibitem{LN2}
Lee, H. and Nungesser, E.
Bianchi I solutions of the Einstein-Boltzmann system with a positive cosmological constant. 
{\it J. Math. Phys.} 58 (2017), no. 9, 092501.

\bibitem{LN3} 
Lee, H. and Nungesser, E.
Late-time behaviour of the Einstein-Boltzmann system with a positive cosmological constant. 
{\it Classical Quantum Gravity} 35 (2018), no. 2, 025001.

 
\bibitem{LNT}
Lee, H., Nungesser, E., and Tod, K. P.
The massless Einstein-Boltzmann system with a conformal gauge singularity in an FLRW background.
{\it Classical Quantum Gravity} 37 (2020), no. 3, 035005.


\bibitem{LR} 
Lee, H. and Rendall, A. D. 
The spatially homogeneous relativistic Boltzmann equation with a hard potential. 
{\it Comm. Partial Differential Equations} 38 (2013), no. 12, 2238--2262.


\bibitem{MW}
Mischler, S. and Wennberg, B.
On the spatially homogeneous Boltzmann equation. 
{\it Ann. Inst. H. Poincar{\'e} Anal. Non Lin{\'e}aire} 16 (1999), no. 4, 467--501.



\bibitem{ND06}
Noutchegueme, N. and Dongo, D.
Global existence of solutions for the Einstein-Boltzmann system in a Bianchi type I spacetime for arbitrarily large initial data.
{\it Classical Quantum Gravity} 23 (2006), no. 9, 2979--3003.

\bibitem{NDT05}
Noutchegueme, N., Dongo, D., and Takou, E.
Global existence of solutions for the relativistic Boltzmann equation with arbitrarily large initial data on a Bianchi type I space-time.
{\it Gen. Relativity Gravitation} 37 (2005), no. 12, 2047--2062.

\bibitem{NT06} 
Noutchegueme, N. and Takou, E.
Global existence of solutions for the Einstein-Boltzmann system with cosmological constant in the Robertson-Walker space-time.
{\it Commun. Math. Sci.} 4 (2006), no. 2, 291--314.






\bibitem{S10}
Strain, R. M. 
Asymptotic stability of the relativistic Boltzmann equation for the soft potentials.
{\it Comm. Math. Phys.} 300 (2010), no.2, 529--597.

\bibitem{SY}
Strain, R. M. and Yun, S.-B.
Spatially homogeneous Boltzmann equation for relativistic particles.
{\it SIAM J. Math. Anal.} 46 (2014), no. 1, 917--938.

\bibitem{SZ}
Strain, R. M. and Zhu, K.
Large-time decay of the soft potential relativistic Boltzmann equation in $\bbr^3_x$.
{\it Kinet. Relat. Models} 5 (2012), no. 2, 383--415.


\bibitem{T03}
Tod~K.~P.
\newblock{Isotropic cosmological singularities: other matter models.}
\newblock{\it Classical Quantum Gravity}, 20 (2003), 521--534.

\bibitem{T}
Tod, K. P.
\newblock{Isotropic cosmological singularities in spatially homogeneous models with a cosmological constant.}
\newblock{\em Classical Quantum Gravity} 24 (2007), no. 9, 2415--2432.



\end{thebibliography}

\end{document}